\documentclass[12pt]{amsart}

\newtheorem{theorem}{Theorem}[section]
\newtheorem{lemma}[theorem]{Lemma}

\theoremstyle{definition}
\newtheorem{definition}[theorem]{Definition}

\newtheorem{proposition}[theorem]{Proposition}

\theoremstyle{remark}
\newtheorem{remark}[theorem]{Remark}

\numberwithin{equation}{section}

\usepackage[colorlinks=true,breaklinks=true,linkcolor=blue,citecolor=red]{hyperref}
\usepackage[letterpaper, left=2.5cm, right=2.5cm, top=2.5cm,
bottom=2.5cm,dvips]{geometry}

\DeclareMathOperator{\diver}{div}

\begin{document}

\title[Parabolic differential inequalities in exterior domains]{Nonexistence for parabolic differential inequalities with convection terms in exterior domains}

%    Remove any unused author tags.

%    author one information
\author[M. Jleli]{Mohamed Jleli}
\address{(M. Jleli) Department of Mathematics, College of Science, King Saud University, Riyadh 11451, Saudi Arabia}
\curraddr{}
\email{jleli@ksu.edu.sa}
\thanks{}

%    author two information
\author[B. Samet]{Bessem Samet}
\address{(B. Samet) Department of Mathematics, College of Science, King Saud University, Riyadh 11451, Saudi Arabia}
\curraddr{}
\email{bsamet@ksu.edu.sa}
\thanks{}

\author[Y. Sun]{Yuhua Sun}
\address{(Y. Sun) School of Mathematical Sciences and LPMC, Nankai University, 300071 Tianjin, 
P.R. China}
\curraddr{}
\email{sunyuhua@nankai.edu.cn}
\thanks{}



\date{}

\dedicatory{}

\begin{abstract}
We are concerned with the nonexistence of sign-changing global weak solutions for a  class of semilinear  parabolic differential inequalities with convection terms in exterior domains. A weight function of the form $t^\alpha |x|^\sigma$  is considered in front of the power nonlinearity. Two types of non-homogeneous boundary conditions are investigated: Neumann-type and Dirichlet-type boundary conditions. Using a unified approach,  for each case, we establish sufficient criteria  for the nonexistence of global weak solutions. When $\alpha=0$, the critical exponent in the sense of Fujita is obtained. This exponent is bigger than that found previously by Zheng and  Wang (2008)  in the case of homogeneous Neumann and Dirichlet boundary conditions. 
\end{abstract}

\maketitle

\tableofcontents

\section{Introduction}\label{sec1}

This paper is devoted to the study of  nonexistence of  global weak solutions to  parabolic differential inequalities of the form
\begin{equation}\label{P}
\partial_t u-\Delta u-k \frac{x}{|x|^2}\cdot \nabla u\geq t^\alpha 
|x|^{\sigma}  |u|^p,\quad (t,x)\in (0,\infty)\times B_1^c,	
\end{equation}
where  $k\in \mathbb{R}$, $\alpha\geq -1$, $\sigma\in \mathbb{R}$ and $p>1$. Here,  $B_1$ denotes the open ball of radius $1$ centered at the origin point in $\mathbb{R}^N$ with $N\geq 2$,  $B_1^c$ denotes the complement of $B_1$ and $\cdot$ is the inner product in $\mathbb{R}^N$.  Problem \eqref{P} is investigated under two types of non-homogeneous boundary conditions: the Neumann type condition
\begin{equation}\label{BC1}
\frac{\partial u}{\partial \nu}(t,x)\geq  f(x),\quad (t,x)\in (0,\infty)\times \partial B_1
\end{equation}
and the Dirichlet-type condition
\begin{equation}\label{BC3}
u(t,x)\geq  f(x),\quad (t,x)\in (0,\infty)\times \partial B_1,
\end{equation}
where $f\in L^1(\partial B_1)$ is a nontrivial function and $\nu$ is the outward unit normal vector on $\partial B_1$, relative to $B_1^c$.  We mention below some motivations for studying the considered problems. 

The large-time behavior of solutions to the Cauchy problem of the semilinear equation
\begin{equation}\label{parabolicwholespace}
\partial_t u-\Delta u =u^p,\quad (t,x)\in (0,\infty)\times \mathbb{R}^N
\end{equation}
was first investigated by Fujita \cite{Fujita}. Namely, it was shown that \eqref{parabolicwholespace} possesses a critical behavior in the following sense:  
\begin{itemize}
\item[(a)] If $1<p<1+\frac{2}{N}$, then \eqref{parabolicwholespace} does not have any nontrivial global positive solution;
\item[(b)] If $p>1+\frac{2}{N}$, then \eqref{parabolicwholespace}  admits global positive solutions with small initial data. 
\end{itemize}
We say that $1+\frac{2}{N}$ is  the Fujita critical exponent.  Later, it was shown that $p=1+\frac{2}{N}$ belongs to  the case (a) (see \cite{HA} for $N=1,2$ and  \cite{KO} for any $N\geq 1$). 

In \cite{Fujita}, Fujita asked the following question: If in problem \eqref{parabolicwholespace} the whole space is replaced by the exterior of a bounded domain, is $1+ \frac{2}{N}$ still the critical exponent? An affirmative answer to this question  was furnished by Bandle and Levine \cite{BL} in the case of an homogeneous Dirichlet boundary condition. Later, Levine and Zhang \cite{LE} obtained the same result in the case of an homogeneous Neumann  boundary condition.

In \cite{Zheng}, Zheng and Wang investigated the large-time behavior of nonnegative solutions to a class of quasilinear parabolic equations with convection terms in exterior domains. The considered class includes as a special case  parabolic equations of the form
\begin{equation}\label{P-Zheng}
\partial_t u-\Delta u-k \frac{x}{|x|^2}\cdot \nabla u =|x|^\sigma u^p, \quad (t,x)\in (0,\infty)\times \mathbb{R}^N\backslash \overline{\omega},	
\end{equation}
where $k\in \mathbb{R}$, $\sigma\geq 0$, $p>1$ and  $\omega$ is a bounded domain in $\mathbb{R}^N$ containing the origin with a smooth boundary $\partial \omega$. Problem \eqref{P-Zheng} was studied under the homogeneous Neumann boundary condition 
\begin{equation}\label{BC1Zheng}
\frac{\partial u}{\partial \nu}(t,x)=0,\quad (t,x)\in (0,\infty)\times \partial \omega 	
\end{equation}
and the homogeneous Dirichlet boundary condition
\begin{equation}\label{BC2Zheng}
u(t,x)=0,\quad (t,x)\in (0,\infty)\times \partial \omega.	
\end{equation}
Namely, for both problems, it was shown that  
$$
p_c=\left\{\begin{array}{llll}
1+\frac{\sigma+2}{N+k}, &\mbox{if}& k>-N,\\
\infty,&\mbox{if}& k\leq -N	
\end{array}
\right.
$$
is critical in the sense of Fujita. Moreover, $p=p_c$ belongs to the blow-up case under any nontrivial initial data. 

In \cite{Zhang}, Zhang considered for the first time the exterior problem 
\begin{equation}\label{P-Zhang}
\partial_t u-\Delta u=|x|^\sigma u^p, \quad (t,x)\in (0,\infty)\times \mathbb{R}^N\backslash \overline{\omega},	
\end{equation}
where $N\geq 3$, $\sigma>-2$ and $p>1$, subject to the non-homogeneous Neumann boundary condition 
\begin{equation}\label{BC1Z}
\frac{\partial u}{\partial \nu}(t,x)= f(x),\quad (t,x)\in (0,\infty)\times \partial\textcolor{red}{\omega}
\end{equation}
or the non-homogeneous Dirichlet boundary condition 
\begin{equation}\label{BC3Z}
u(t,x)=  f(x),\quad (t,x)\in (0,\infty)\times \partial \textcolor{red}{\omega},
\end{equation}
where  $f\geq 0$ is a nontrivial $L^1(\partial \omega)$ function. An interesting phenomenon was observed in this case. Namely, it was shown that the critical exponent in the sense of Fujita  for problems \eqref{P-Zhang}-\eqref{BC1Z} and  \eqref{P-Zhang}-\eqref{BC3Z} jumps from  $1+\frac{\sigma+2}{N}$ (i.e. the critical exponent for problems  \eqref{P-Zhang}-\eqref{BC1Zheng} and  \eqref{P-Zhang}-\eqref{BC2Zheng}) to the bigger exponent $1+\frac{\sigma+2}{N-2}$.  In the case $N=2$ and $\sigma=0$,  it was shown in \cite{JS19} that  the critical exponent   for problems  \eqref{P-Zhang}-\eqref{BC1Z} and  \eqref{P-Zhang}-\eqref{BC3Z} is equal to $\infty$.  

In \cite{Sun17},  the third author considered  parabolic differential inequalities of the form 
\begin{equation}\label{P-Sun}
\partial_t u-\Delta u\geq t^\alpha |x|^{\sigma}  u^p,\quad (t,x)\in (0,\infty)\times B_1^c,	
\end{equation}
where $\alpha>-1$, $\sigma>-2$ and $N\geq 3$. Problem \eqref{P-Sun} was investigated subject to the non-homogeneous Dirichlet-type boundary condition \eqref{BC3}, where $f\in C^1(\partial B_1)$, $f\geq 0$ and $f\not\equiv 0$. Namely, for nonnegative initial data,  it was shown that, if
\begin{equation}\label{cd-Sun}
1<p<1+\frac{\sigma+2+4\alpha N}{N-2},
\end{equation}
then problem \eqref{P-Sun} subject to the boundary condition \eqref{BC3} admits no global weak solution.  Clearly, from \eqref{cd-Sun}, it was supposed implicitly that  $\sigma>-(2+4\alpha N)$. Observe that in the case of nonnegative solutions, \eqref{P-Sun} is a special case of \eqref{P} with $k=0$. 

For additional results related to differential inequalities in unbounded domains, see e.g. \cite{LA} in the case of cone-like domains and \cite{GRI14,MA,WA} in the case of manifolds.

The novelty of our work lies on: 
\begin{itemize}
\item[(1)]	the investigation of large-time behavior of sign-changing solutions to \eqref{P};
\item[(2)] the consideration of non-homogeneous boundary conditions.
\end{itemize}
Notice that in our case, the methods used in \cite{Sun17,Zheng} cannot be applied due to the change of signs of solutions.  In this paper, our approach is based on the nonlinear capacity method developed by Mitidieri and Pohozaev \cite{MP} (see also \cite{BI,CA,DA}) and the use of adequate test functions that depend on the considered boundary conditions.

Before stating our main results, let us mention in which sense the solutions are considered. 
Let 
$$
\Omega=(0,\infty)\times B_1^c\quad\mbox{and}\quad \Gamma=(0,\infty)\times\partial B_1. 
$$

For the Neumann type boundary condition, we introduce the test function space
$$
\Phi=\left\{\varphi\in C_{t,x}^{1,2}(\Omega):\, \varphi\geq 0,\, \mbox{supp}(\varphi)\subset\subset \Omega,\,  \left(\frac{\partial \varphi}{\partial \nu}+k\varphi\right)|_{\Gamma}= 0\right\}.
$$
Notice that $B_1^c$ is closed and $\Gamma\subset \Omega$. 
\begin{definition}\label{defws1}
We say that $u\in L^p_{loc}(\Omega)$ is a global weak solution to problem \eqref{P} under the boundary condition \eqref{BC1}, if 
\begin{equation}\label{wsP-BC1}
\begin{aligned}
&\int_\Omega t^\alpha |x|^{\sigma}|u|^p \varphi \,dx\,dt	+\int_\Gamma f\varphi\, dS_x\,dt\\
& \leq - \int_\Omega u \partial_t \varphi\,dx\,dt +\int_\Omega \left(-\Delta \varphi+
k\diver\left(\varphi \frac{x}{|x|^2}\right)\right)u \,dx\,dt,
\end{aligned} 
\end{equation}
 for every $\varphi\in \Phi$. 
\end{definition}

For the Dirichlet-type boundary condition, we introduce the test function space
$$
\Psi=\left\{\psi\in C_{t,x}^{1,2}(\Omega):\, \psi\geq 0,\, \mbox{supp}(\psi)\subset\subset \Omega,\, \psi|_{\Gamma}=0,\, \frac{\partial \psi}{\partial \nu}|_{\Gamma}\leq  0\right\}.
$$

\begin{definition}\label{defws3}
We say that $u\in L^p_{loc}(\Omega)$ is a global weak solution to problem \eqref{P} under the boundary condition \eqref{BC3}, if 
\begin{equation}\label{wsP-BC3}
\begin{aligned}
&\int_\Omega t^\alpha |x|^{\sigma}|u|^p \psi \,dx\,dt	-\int_\Gamma  f\frac{\partial \psi}{\partial \nu}\, dS_x\,dt\\
& \leq -\int_\Omega u \partial_t\psi\,dx\,dt +\int_\Omega \left(-\Delta \psi+k\diver\left(\psi \frac{x}{|x|^2}\right)\right)u \,dx\,dt,
\end{aligned} 
\end{equation}
 for every $\psi\in \Psi$. 
\end{definition}

For $f\in L^1(\partial B_1)$, let 
$$
I_f=\int_{\partial B_1}f(x)\,dS_x.
$$
In the case $\alpha>0$, we have the following result.

\begin{theorem}\label{T1}
Assume that $f\in L^1(\partial B_1)$ and $I_f>0$. If $\alpha>0$, then 	for all $p>1$, 
\begin{itemize}
\item[(i)] problem \eqref{P} under the boundary condition \eqref{BC1} admits no global weak solution;
\item[(ii)] problem \eqref{P} under the boundary condition \eqref{BC3} admits no global weak solution.
\end{itemize}
\end{theorem}

Now, we consider the case $-1\leq \alpha\leq 0$. We first suppose that $k>2-N$. 

\begin{theorem}\label{T2}
Assume that $f\in L^1(\partial B_1)$ and $I_f>0$. If $-1\leq \alpha\leq 0$,  	$\sigma+2(\alpha+1)>0$ and $k>2-N$, then for all 
\begin{equation}\label{blowup-cd}
1<p<1+\frac{\sigma+2(\alpha+1)}{N+k-2},	
\end{equation}
\begin{itemize}
\item[(i)] problem \eqref{P} under the boundary condition \eqref{BC1} admits no global weak solution;
\item[(ii)] problem \eqref{P} under the boundary condition \eqref{BC3} admits no global weak solution.
\end{itemize}
\end{theorem}

Next, we suppose that $k\leq 2-N$. 

\begin{theorem}\label{T3}
Assume that $f\in L^1(\partial B_1)$ and $I_f>0$. If $-1\leq \alpha\leq 0$ and $k\leq 2-N$, then for all  $p>1$, 
\begin{itemize}
\item[(i)] problem \eqref{P} under the boundary condition \eqref{BC1} admits no global weak solution for $\sigma+2(\alpha+1)>2-N-k$; 
\item[(ii)] problem \eqref{P} under the boundary condition \eqref{BC3} admits no global weak solution for  	$\sigma+2(\alpha+1)>0$. 
\end{itemize}
\end{theorem}

\begin{remark}
Let us consider problem \eqref{P-Sun}  under the Dirichlet-type boundary condition \eqref{BC3}, where $f\in C^1(\partial B_1)$, $f\geq 0$ and $f\not\equiv 0$. As we mentioned in Section \ref{sec1}, it was shown in \cite{Sun17} that for nonnegative initial data, if 
\begin{equation}\label{cds-Sun17}
\alpha>-1,\,\, \sigma>\max\left\{-2, -(2+4\alpha N)\right\},\,\, N\geq 3,\,\,1<p<1+\frac{\sigma+2+4\alpha N}{N-2},
\end{equation}
then problem \eqref{P-Sun} subject to the boundary condition \eqref{BC3} admits no global weak solution. On the other hand, by Theorems \ref{T1}, \ref{T2} and \ref{T3}, we deduce that, if  $f\in L^1(\partial B_1)$, $I_f>0$ and
\begin{equation}\label{cds-JSS}
\left\{\begin{array}{llll}
\alpha>0,\,\,  p>1 \\
\mbox{ or}\\
-1\leq \alpha\leq 0,\,\, \sigma>-2(\alpha+1),\,\, N=2,\,\, p>1\\
\mbox{ or}\\
-1\leq \alpha\leq 0,\,\, \sigma>-2(\alpha+1),\,\, N\geq 3,\,\,1<p<1+\displaystyle\frac{\sigma+2(\alpha+1)}{N-2},
\end{array}
\right.
\end{equation}
 then problem \eqref{P-Sun} subject to the boundary condition \eqref{BC3} admits no global weak solution. Comparing \eqref{cds-Sun17} with \eqref{cds-JSS}, we see that our result improves that obtained  in \cite{Sun17}. Notice also that no condition on the initial data is imposed in our results and the  considered class of functions $f$  is more large than that in \cite{Sun17}. 
 \end{remark}

Finally, we consider the critical case when $\alpha=0$ and $k>2-N$. 

\begin{theorem}\label{T4}
Assume that $f\in L^1(\partial B_1)$ and $I_f>0$. If $\alpha=0$,  	$\sigma>-2$ and $k>2-N$, then for 
\begin{equation}\label{blowup-cd-cr}
p=1+\frac{\sigma+2}{N+k-2},	
\end{equation}
\begin{itemize}
\item[(i)] problem \eqref{P} under the boundary condition \eqref{BC1} admits no global weak solution;
\item[(ii)] problem \eqref{P} under the boundary condition \eqref{BC3} admits no global weak solution.
\end{itemize}
\end{theorem}

\begin{remark}
At this time, for $-1\leq \alpha<0$,  	$\sigma+2(\alpha+1)>0$ and $k>2-N$,	we do not know whether condition \eqref{blowup-cd} is sharp or not. However, for $\alpha=0$, $\sigma>-2$ and $k>2-N$, the exponent $1+\frac{\sigma+2}{N+k-2}$ is the Fujita critical exponent for problems \eqref{P}-\eqref{BC1} and \eqref{P}-\eqref{BC3}.  Namely, if $p>1+\frac{\sigma+2}{N+k-2}$, taking $u(x)=\varepsilon |x|^\delta$, where 
$$
2-N-k<\delta<-\frac{\sigma+2}{p-1}<0
$$
and 
$$
0<\varepsilon< \left[\delta(2-N-k-\delta)\right]^{\frac{1}{p-1}},
$$
we can check that $u$ is a stationary solution to problems \eqref{P}-\eqref{BC1} and \eqref{P}-\eqref{BC3} for suitable $f\geq 0$. 
\end{remark}

\begin{remark}
It is interesting to observe that in the case $\alpha=0$, $\sigma>-2$ and $k>2-N$, the critical exponent for problems \eqref{P}-\eqref{BC1} and \eqref{P}-\eqref{BC3} jumps from $p_c=1+\frac{\sigma+2}{N+k}	$ (i.e. the critical exponent for \eqref{P-Zheng}-\eqref{BC1Zheng} and \eqref{P-Zheng}-\eqref{BC2Zheng}) to the bigger exponent $1+\frac{\sigma+2}{N+k-2}$.
\end{remark}

The rest of the paper is organized as follows. Section \ref{sec2} is devoted to some preliminary estimates. Namely, we first establish a priori estimates for global weak solutions to problems \eqref{P}-\eqref{BC1} and \eqref{P}-\eqref{BC3}. Next, according to the considered boundary conditions, some test functions are introduced and some useful estimates are established.   Finally, we prove our main results in Section \ref{sec3}. 

 Throughout this paper,  the letter $C$  denotes always a generic positive constant whose value is unimportant and may vary at different occurrences. For $\delta>0$, we denote by $B_\delta$ the open ball of radius $\delta$ centered at the origin point in $\mathbb{R}^N$.  For $0<\delta_1<\delta_2$, we denote by $C(\delta_1,\delta_2)$ the subset of $\mathbb{R}^N$ defined by 
$$
C(\delta_1,\delta_2)=\left\{x\in \mathbb{R}^N:\, \delta_1<|x|<\delta_2\right\}.
$$

\section{Preliminaries} \label{sec2}

\subsection{A priori estimates}\label{s2.1}

The following a priori estimates will play a crucial role in the proofs of our main results. 

For $k\in \mathbb{R}$,  let $L_k$ be the differential operator defined by 
\begin{equation}\label{DF-LK}
L_k v=\Delta v -k\diver\left(\frac{v}{|x|^2}x\right).
\end{equation}
For $\mu\in \Phi\cup \Psi$, let 
\begin{equation}\label{I1phi}
I_1(\mu)=\int_\Omega t^{\frac{-\alpha}{p-1}}|x|^{\frac{-\sigma}{p-1}}\mu^{\frac{-1}{p-1}}|\partial_t\mu|^{\frac{p}{p-1}}\,dx\,dt
\end{equation}
and
\begin{equation}\label{I2phi}
I_2(\mu)=\int_\Omega t^{\frac{-\alpha}{p-1}}|x|^{\frac{-\sigma }{p-1}} \mu^{\frac{-1}{p-1}} |L_k\mu|^{\frac{p}{p-1}}\,dx\,dt.
\end{equation}

\begin{proposition}\label{PR1}
Assume that $u\in L^p_{loc}(\Omega)$ is a global weak solution to problem \eqref{P} under the boundary condition \eqref{BC1}. 	Then
\begin{equation}\label{apest1}
\int_\Gamma  f\varphi\, dS_x\,dt\leq C \sum_{i=1}^2 I_i(\varphi),	
\end{equation}
for every $\varphi\in \Phi$,  provided that  $I_i(\varphi)<\infty$, $i=1,2$.  
\end{proposition}

\begin{proof}
By \eqref{wsP-BC1}, we have
\begin{equation}\label{gestP1}
\int_\Omega t^\alpha |x|^{\sigma}|u|^p \varphi \,dx\,dt	+\int_\Gamma  f\varphi\, dS_x\,dt
 \leq  \int_\Omega |u| |\partial_t\varphi|\,dx\,dt +\int_\Omega |L_k\varphi| |u| \,dx\,dt,
\end{equation}
 for every $\varphi\in \Phi$. On the other hand, using $\varepsilon$-Young inequality with $\varepsilon>0$, we get
 \begin{equation}\label{est1P1}
 \int_\Omega |u| |\partial_t\varphi|\,dx\,dt \leq \varepsilon \int_\Omega t^\alpha |x|^{\sigma}|u|^p \varphi \,dx\,dt	+C I_1(\varphi)	
 \end{equation}
and
\begin{equation}\label{est2P1}
\int_\Omega |L_k\varphi| |u| \,dx\,dt\leq \varepsilon \int_\Omega t^\alpha |x|^{\sigma}|u|^p \varphi \,dx\,dt+C I_2(\varphi).		
\end{equation}
In view of \eqref{gestP1}, \eqref{est1P1} and \eqref{est2P1}, there holds
$$
(1-2\varepsilon) \int_\Omega t^\alpha |x|^{\sigma}|u|^p \varphi \,dx\,dt+\int_\Gamma  f\varphi\, dS_x\,dt\leq C \sum_{i=1}^2 I_i(\varphi).
$$
Taking $\varepsilon=\frac{1}{2}$ in the above inequality, we obtain \eqref{apest1}.
\end{proof}

 Following the same argument used in the proof of Proposition \ref{PR1}, we obtain the following a priori estimate for problem \eqref{P} under the boundary condition \eqref{BC3}. 

\begin{proposition}\label{PR2}
Assume that $u\in L^p_{loc}(\Omega)$ is a global weak solution to problem \eqref{P} under the boundary condition \eqref{BC3}. 	Then
\begin{equation}\label{apest2}
-\int_\Gamma  f \frac{\partial \psi}{\partial \nu}\, dS_x\,dt\leq C \sum_{i=1}^2 I_i(\psi),
\end{equation}
for every $\psi\in \Psi$,  provided that  $I_i(\psi)<\infty$, $i=1,2$.  
\end{proposition}

\subsection{Test functions}
Let  $\mathcal{N}$ be the function defined by
$$
\mathcal{N}(x)=|x|^{k},\quad x\in B_1^c.
$$
One can check easily that the function $\mathcal{N}$ satisfies
\begin{equation}\label{LKN}
\left\{\begin{array}{llll}
L_k \mathcal{N}= 0 &\mbox{ in }& B_1^c,\\ 
\displaystyle\frac{\partial \mathcal{N}}{\partial \nu}+k\mathcal{N} = 0 	&\mbox{ on }& \partial B_1,
 \end{array}
\right.
\end{equation}
where $L_k$ is the differential operator defined by \eqref{DF-LK}.  

Let $\mathcal{D}$ be the function defined in $B_1^c$ by 
$$
\mathcal{D}(x)=\left\{\begin{array}{llll}
|x|^k\left(1-|x|^{2-N-k}\right), &\mbox{if}& k>2-N,\\ 
|x|^{2-N}\left(1-|x|^{k-2+N}\right), &\mbox{if}& k<2-N,\\ 
|x|^{2-N}\ln |x|, &\mbox{if}& k=2-N.
\end{array}
\right.
$$
It can be easily seen that the function $\mathcal{D}$ is nonnegative and satisfies
\begin{equation}\label{LKD}
\left\{\begin{array}{llll}
L_k \mathcal{D} = 0 &\mbox{ in }& B_1^c,\\ 
\mathcal{D} = 0 	&\mbox{ on }& \partial B_1.
 \end{array}
\right.
\end{equation}

Let $\zeta\in C_c^\infty(\mathbb{R})$ be such that 
\begin{equation}\label{pptszeta}
\zeta\geq 0,\quad \zeta\not\equiv 0,\quad  \mbox{supp}(\zeta)\subset (0,1).
\end{equation}
For sufficiently large $T$ and $\ell$, we introduce the function  
$$
\iota(t)=\zeta^\ell\left(\frac{t}{T}\right),\quad t>0. 
$$

For sufficiently large $R$, let  $\{\xi_R\}_R\subset C_c^\infty(\mathbb{R}^N)$ be  a family of cut-off functions satisfying 
\begin{equation}\label{pptsxiR}
0\leq \xi_R\leq 1,\quad \xi_R\equiv 1\mbox{ in } B_R,\quad \mbox{supp}(\xi_R)\subset B_{2R},\quad |\nabla \xi_R|\leq C R^{-1},\quad |\Delta \xi_R|\leq CR^{-2}.
\end{equation}
We need also a family of cut-off functions $\{\eta_R\}_R\subset C_c^\infty(\mathbb{R}^N)$ satisfying 
\begin{equation}\label{pptsetaRI}
0\leq \eta_R\leq 1,\quad \eta_R\equiv 1\mbox{ in } B_{\sqrt R},\quad \mbox{supp}(\eta_R)\subset B_R
\end{equation}
and
\begin{equation}\label{pptsetaRII}
|\nabla \eta_R|\leq \frac{C}{|x|\ln R},\quad |\Delta \eta_R|\leq \frac{C}{|x|^2\ln R},\quad x\in C(\sqrt R,R).		
\end{equation}
For instance, taking a smooth function $\vartheta: \mathbb{R}\to [0,1]$ satisfying
$$
\vartheta(s)=\left\{\begin{array}{llll}
1, &\mbox{if}& s\leq 0,	\\
0,&\mbox{if}& s\geq 1,
\end{array}
\right.
$$
one can check easily that the family of functions $\{\eta_R\}_R$ defined by
$$
\eta_R(x)=\vartheta\left(\frac{\ln |x|}{\ln\sqrt{R}}-1\right),\quad x\in \mathbb{R}^N,
$$
satisfies \eqref{pptsetaRI} and \eqref{pptsetaRII}.

For problem \eqref{P} under the boundary condition \eqref{BC1}, we first introduce test functions of the form
\begin{equation}\label{testfP11}
\varphi(t,x)=\iota(t) F(x),\quad (t,x)\in \Omega,	
\end{equation}
where
$$
F(x)=\mathcal{N}(x) \xi_R^\ell(x),\quad x\in B_1^c. 
$$

\begin{lemma}\label{L1}
For sufficiently large $T$, $\ell$ and $R$, the function $\varphi$ defined by \eqref{testfP11} belongs to the test function space 	$\Phi$.
\end{lemma}

\begin{proof}
We have just to show that 	$\left(\frac{\partial \varphi}{\partial \nu}+k\varphi\right)|_{\Gamma}= 0$.  In view of \eqref{testfP11}, we have 
\begin{equation}\label{obser1}
\frac{\partial \varphi}{\partial \nu}(t,x)+k\varphi(t,x)=\iota(t) \left(\frac{\partial F}{\partial \nu}(x)+kF(x)\right),\quad (t,x)\in \Gamma.
\end{equation}
On the other hand, for $x\in C(1,R)$, since $\xi_R\equiv 1\mbox{ in } B_R$, we get
\begin{eqnarray*}
\nabla F(x)&=&\nabla \left(\mathcal{N}(x) \xi_R^\ell(x)\right)	\\
&=& \xi_R^\ell(x)\nabla \mathcal{N}(x)+\mathcal{N}(x) \nabla (\xi_R^\ell(x))\\
&=& \nabla \mathcal{N}(x).
\end{eqnarray*}
Hence, by \eqref{LKN}, there holds
$$
\frac{\partial F}{\partial \nu}(x)+kF(x)=\frac{\partial \mathcal{N}}{\partial \nu}(x)+k\mathcal{N}(x)=0,\quad x\in \partial B_1,
$$
which implies by  \eqref{obser1} that
$$
\frac{\partial \varphi}{\partial \nu}(t,x)+k\varphi(t,x)=0,\quad (t,x)\in \Gamma.
$$
\end{proof}

Next, we introduce test functions of the form
\begin{equation}\label{testfP11-cr}
\varphi^*(t,x)=\iota(t) F^*(x),\quad (t,x)\in \Omega,	
\end{equation}
where
$$
F^*(x)=\mathcal{N}(x) \eta_R^\ell(x),\quad x\in B_1^c. 
$$
\begin{lemma}\label{L2}
For sufficiently large $T$, $\ell$ and $R$, the function $\varphi^*$ defined by \eqref{testfP11-cr} belongs to the test function space 	$\Phi$.
\end{lemma}

\begin{proof}
We have just to show that 	$\left(\frac{\partial \varphi^*}{\partial \nu}+k\varphi^*\right)|_{\Gamma}= 0$.  	In view of \eqref{testfP11-cr}, we have 
$$
\frac{\partial \varphi^*}{\partial \nu}(t,x)+k\varphi^*	(t,x)=\iota(t) \left(\frac{\partial F^*}{\partial \nu}(x)+kF^*(x)\right),\quad (t,x)\in \Gamma.
$$
On the other hand, for $x\in C(1,\sqrt R)$, since $\eta_R\equiv 1\mbox{ in } B_{\sqrt R}$, we get
$$
\nabla F^*(x)=\nabla \mathcal{N}(x).
$$
The rest of the proof is the same as that of Lemma \ref{L1}.
\end{proof}

For problem \eqref{P} under the boundary condition \eqref{BC3}, we first introduce test functions of the form
\begin{equation}\label{testfP12}
\psi(t,x)=\iota(t) G(x),\quad (t,x)\in \Omega,	
\end{equation}
where
$$
G(x)=\mathcal{D}(x) \xi_R^\ell(x),\quad x\in B_1^c. 
$$

\begin{lemma}\label{L3}
For sufficiently large $T$, $\ell$ and $R$, the function $\psi$ defined by \eqref{testfP12} belongs to the test function space 	$\Psi$.
\end{lemma}

\begin{proof}
We have just to show that $\frac{\partial \mathcal{\psi}}{\partial \nu}|_{\Gamma}\leq 0$.	 In view of \eqref{testfP12}, we have
\begin{equation}\label{observation1}
\frac{\partial \psi}{\partial \nu}(t,x)=\iota(t) \frac{\partial G}{\partial \nu}(x),\quad (t,x)\in \Gamma.
\end{equation}
On the other hand, for $x\in C(1,R)$, since $\xi_R\equiv 1\mbox{ in } B_R$, we get
\begin{eqnarray*}
\nabla G(x)&=& \nabla \mathcal{D}(x)\\
&=& \left\{\begin{array}{llll}
\left(k|x|^{k-2}-(2-N)|x|^{-N}\right)x, &\mbox{if}& k>2-N,\\ 
\left((2-N)|x|^{-N}-k|x|^{k-2}\right)x, &\mbox{if}& k<2-N,\\ 
\left((2-N)|x|^{-N}\ln |x|+|x|^{-N}\right)x, &\mbox{if}& k=2-N.
\end{array}
\right.
\end{eqnarray*}
Hence, for $x\in \partial B_1$, there holds
$$
\frac{\partial G}{\partial \nu}(x)=\left\{\begin{array}{llll}
2-N-k, &\mbox{if}& k>2-N,\\ 
k+N-2, &\mbox{if}& k<2-N,\\ 
-1, &\mbox{if}& k=2-N.
\end{array}
\right.
$$
Thus, in view of \eqref{observation1} (notice that $\iota(t)\geq 0$), we obtain
\begin{equation}\label{useful-Dirichlet-ND}
\frac{\partial \psi}{\partial \nu}(t,x)=\left\{\begin{array}{llll}
(2-N-k)\iota(t)\leq 0, &\mbox{if}& k>2-N,\\ 
(k+N-2)\iota(t)\leq 0, &\mbox{if}& k<2-N,\\ 
-\iota(t)\leq 0, &\mbox{if}& k=2-N.
\end{array}
\right.
\end{equation}
\end{proof}

Next, we introduce test functions of the form
\begin{equation}\label{testfP12-cr}
\psi^*(t,x)=\iota(t) G^*(x),\quad (t,x)\in \Omega,	
\end{equation}
where
$$
G^*(x)=\mathcal{D}(x) \eta_R^\ell(x),\quad x\in B_1^c. 
$$

\begin{lemma}\label{L4}
For sufficiently large $T$, $\ell$ and $R$, the function $\psi^*$ defined by \eqref{testfP12-cr} belongs to the test function space 	$\Psi$.
\end{lemma}

\begin{proof}
Proceeding as in the proof of Lemma \ref{L3}, in view of \eqref{testfP12-cr} and using that 	$\eta_R\equiv 1\mbox{ in } B_{\sqrt R}$, we get
\begin{equation}\label{useful-Dirichlet-ND-cr}
\frac{\partial \psi^*}{\partial \nu}(t,x)=\left\{\begin{array}{llll}
(2-N-k)\iota(t)\leq 0, &\mbox{if}& k>2-N,\\ 
(k+N-2)\iota(t)\leq 0, &\mbox{if}& k<2-N,\\ 
-\iota(t)\leq 0, &\mbox{if}& k=2-N.
\end{array}
\right.
\end{equation}
\end{proof}

\subsection{Useful estimates}

In this subsection, we shall estimate the terms $I_i(\mu)$, $i=1,2$, for $\mu\in\{\varphi,\varphi^*,\psi,\psi^*\}$, where the test functions $\varphi$, $\varphi^*$, $\psi$ and $\psi^*$ are defined respectively by \eqref{testfP11}, \eqref{testfP11-cr}, \eqref{testfP12} and \eqref{testfP12-cr}. Such estimates will play a crucial role in the proofs of our main results. 

\begin{lemma}\label{L2.7}
Let $\mu\in \{\varphi,\varphi^*\}$. 	For sufficiently large $T$, $\ell$ and $R$, there holds
\begin{equation}\label{estI1phi-phistar}
I_1(\mu)\leq C T^{1-\frac{\alpha+p}{p-1}}\left(\ln R+ R^{N+k-\frac{\sigma}{p-1}}\right).
\end{equation}
\end{lemma}

\begin{proof}
We shall prove \eqref{estI1phi-phistar} in the case $\mu=\varphi$. The case $\mu=\varphi^*$ can be treated in a similar way.  In view of \eqref{I1phi} and \eqref{testfP11}, we have
\begin{equation}\label{estI1phi-phiN}
I_1(\varphi)=\left(\int_0^\infty t^{\frac{-\alpha}{p-1}} \iota^{\frac{-1}{p-1}}(t) |\iota'(t)|^{\frac{p}{p-1}}\,dt\right)\left(\int_{B_1^c} |x|^{\frac{-\sigma}{p-1}}F(x)\,dx\right).
\end{equation}
 On the other hand, one has
 $$
|\iota'(t)|\leq C T^{-1} \zeta^{\ell-1}\left(\frac{t}{T}\right),\quad t>0,
$$	
which yields
$$
\iota^{\frac{-1}{p-1}}(t)|\iota'(t)|^{\frac{p}{p-1}}\leq C T^{\frac{-p}{p-1}}\zeta^{\ell-\frac{p}{p-1}}\left(\frac{t}{T}\right),\quad t>0.
$$
Hence, by \eqref{pptszeta}, we obtain
\begin{eqnarray*}
\int_0^\infty  t^{\frac{-\alpha}{p-1}}\iota^{\frac{-1}{p-1}}(t)|\iota'(t)|^{\frac{p}{p-1}}\,dt &\leq & C T^{\frac{-p}{p-1}} \int_0^T 	t^{\frac{-\alpha}{p-1}}\zeta^{\ell-\frac{p}{p-1}}\left(\frac{t}{T}\right)\,dt\\
&=& CT^{1-\frac{\alpha+p}{p-1}}\int_{\mbox{supp}(\zeta)} s^{\frac{-\alpha}{p-1}}\zeta^{\ell-\frac{p}{p-1}}(s)\,ds.
\end{eqnarray*}	
 Notice that, since $\mbox{supp}(\zeta)\subset\subset (0,1)$, then $\displaystyle\int_{\mbox{supp}(\zeta)} s^{\frac{-\alpha}{p-1}}\xi^{\ell-\frac{p}{p-1}}(s)\,ds<\infty$ (for sufficiently large $\ell$). Consequently, we get
\begin{equation}\label{timeest-I1}
\int_0^\infty  t^{\frac{-\alpha}{p-1}}\iota^{\frac{-1}{p-1}}(t)|\iota'(t)|^{\frac{p}{p-1}}\,dt\leq C 	T^{1-\frac{\alpha+p}{p-1}}.
\end{equation}
Moreover, by the definition of the function $F$ and using \eqref{pptsxiR},  we obtain
\begin{eqnarray}\label{spaceest-I1}
\nonumber \int_{B_1^c} |x|^{\frac{-\sigma}{p-1}}F(x)\,dx&=& 	\int_{B_1^c} |x|^{\frac{-\sigma}{p-1}+k} \xi_R^\ell(x)\,dx\\
\nonumber &\leq & \int_{C(1,2R)} |x|^{\frac{-\sigma}{p-1}+k} \,dx\\
\nonumber &=& C \int_{r=1}^{2R} r^{N+k-1-\frac{\sigma}{p-1}}\,dr\\
&\leq & C \left(\ln R+ R^{N+k-\frac{\sigma}{p-1}}\right).
\end{eqnarray}
Then, \eqref{estI1phi-phistar} follows from  \eqref{estI1phi-phiN}, \eqref{timeest-I1} and \eqref{spaceest-I1}.
\end{proof}

\begin{lemma}\label{L2.8}
For sufficiently large $T$, $\ell$ and $R$, there holds	
\begin{equation}\label{estI2phi}
I_2(\varphi)\leq C 	T^{1-\frac{\alpha }{p-1}}R^{\frac{(k-2)p-(\sigma+k)}{p-1}+N}.
\end{equation}
Moreover, if  $\alpha=0$ and $p(N+k-2)=N+\sigma+k$, then 
\begin{equation}\label{estI2phistar}
I_2(\varphi^*)\leq CT (\ln R)^{\frac{-1}{p-1}}. 
\end{equation}
\end{lemma}

\begin{proof}
In view of \eqref{DF-LK}, \eqref{I2phi}  and \eqref{testfP11}, we get 
\begin{equation}\label{ayacontinue}
I_2(\varphi)=\left(\int_0^\infty t^{\frac{-\alpha}{p-1}}\iota(t)\,dt\right)\left(\int_{B_1^c} |x|^{\frac{-\sigma}{p-1}} F(x)^{\frac{-1}{p-1}}|L_kF(x)|^{\frac{p}{p-1}}\,dx\right).
\end{equation}
On the other hand, we have
\begin{eqnarray*}
\int_0^\infty t^{\frac{-\alpha}{p-1}} \iota(t)\,dt&=& \int_0^T 	t^{\frac{-\alpha}{p-1}} \zeta^\ell\left(\frac{t}{T}\right)\,dt\\
&=& T^{1-\frac{\alpha}{p-1}}\int_{\mbox{supp}(\zeta)}s^{\frac{-\alpha}{p-1}} \zeta^\ell(s)\,ds,
\end{eqnarray*}
which yields
\begin{equation}\label{OKvasi}
\int_0^\infty t^{\frac{-\alpha}{p-1}} \iota(t)\,dt=C T^{1-\frac{\alpha}{p-1}}.
\end{equation}	
By the definition of the function $F$ and using \eqref{pptsxiR},  we obtain 
\begin{equation}\label{integsupportCR2R}
\int_{B_1^c} |x|^{\frac{-\sigma}{p-1}} F^{\frac{-1}{p-1}}(x)|L_kF(x)|^{\frac{p}{p-1}}\,dx=\int_{C(R,2R)} |x|^{\frac{-\sigma}{p-1}} F^{\frac{-1}{p-1}}(x)|L_kF(x)|^{\frac{p}{p-1}}\,dx.	
\end{equation}
Moreover,  by \eqref{DF-LK} and \eqref{LKN}, elementary calculations show that 
$$
L_kF(x)= \mathcal{N}(x)\Delta(\xi_R^\ell(x))+2\ell \xi_R^{\ell-1}(x)\nabla  \mathcal{N}(x)\cdot \nabla \xi_R(x)-k \ell \xi_R^{\ell-1}(x) \mathcal{N}(x)\nabla \xi_R(x)\cdot \frac{x}{|x|^2},
$$
for every $x\in C(R,2R)$. Then, using Cauchy-Schwarz inequality, \eqref{pptsxiR} and  the property
$$
\Delta (\xi_R^\ell)=\ell(\ell-1)\xi_R^{\ell-2}|\nabla \xi_R|^2+\ell\xi_R^{\ell-1}\Delta \xi_R,
$$
we deduce that
\begin{eqnarray*}
|L_kF(x)|&\leq &  \mathcal{N}(x) |\Delta(\xi_R^\ell(x))|+C \xi_R^{\ell-1}(x)|\nabla  \mathcal{N}(x)| |\nabla \xi_R(x)|+C \xi_R^{\ell-1}(x)\mathcal{N}(x)|\nabla \xi_R(x)||x|^{-1}\\
&=& |x|^k |\Delta(\xi_R^\ell(x))|+C\xi_R^{\ell-1}(x)|x|^{k-1}|\nabla \xi_R(x)|\\
&\leq & C R^{-2}|x|^k \xi_R^{\ell-2}(x)+C R^{-1}|x|^{k-1}\xi_R^{\ell-1}(x)\\
&\leq & C \xi_R^{\ell-2}(x) \left(R^{-2}|x|^k+R^{-1}|x|^{k-1}\right)\\
&\leq & C R^{k-2}\xi_R^{\ell-2}(x),
\end{eqnarray*}
which yields (in view of \eqref{integsupportCR2R} and using that $0\leq \xi_R\leq 1$)
\begin{eqnarray}\label{fisa3}
\nonumber \int_{B_1^c} |x|^{\frac{-\sigma}{p-1}} F^{\frac{-1}{p-1}}(x)|L_kF(x)|^{\frac{p}{p-1}}\,dx &\leq & C R^{\frac{(k-2)p-(\sigma+k)}{p-1}} \int_{C(R,2R)} \xi_R^{\ell-\frac{2p}{p-1}}(x)\,dx\\
&\leq & C R^{\frac{(k-2)p-(\sigma+k)}{p-1}+N}. 
\end{eqnarray}
Thus, by \eqref{ayacontinue},  	\eqref{OKvasi} and \eqref{fisa3}, we obtain \eqref{estI2phi}.

Now, we consider the case 
\begin{equation}\label{assumptsalpha}
\alpha=0,\,\, p(N+k-2)=N+\sigma+k.	
\end{equation}
By \eqref{DF-LK}, \eqref{I2phi}  and \eqref{testfP11-cr}, there holds 
\begin{equation}\label{ayacontinue*}
I_2(\varphi^*)=\left(\int_0^\infty \iota(t)\,dt\right)\left(\int_{B_1^c} |x|^{\frac{-\sigma}{p-1}} {F^*}^{\frac{-1}{p-1}}(x)|L_kF^*(x)|^{\frac{p}{p-1}}\,dx\right).
\end{equation}
By the definition of the function $F^*$ and using \eqref{pptsetaRI}, we get
\begin{equation}\label{integsupportCR2R*}
\int_{B_1^c} |x|^{\frac{-\sigma}{p-1}} {F^*}^{\frac{-1}{p-1}}(x)|L_kF^*(x)|^{\frac{p}{p-1}}\,dx=\int_{C(\sqrt R,R)} |x|^{\frac{-\sigma}{p-1}} {F^*}^{\frac{-1}{p-1}}(x)|L_kF^*(x)|^{\frac{p}{p-1}}\,dx.	
\end{equation}
As in the proof of \eqref{estI2phi}, by \eqref{DF-LK}, \eqref{LKN} and \eqref{pptsetaRII}, for $x\in C(\sqrt R,R)$, we obtain
$$
|L_kF^*(x)|\leq C \frac{|x|^{k-2}}{\ln R} \eta_R^{\ell-2}(x), 
$$
which yields (in view of \eqref{integsupportCR2R*} and using that $0\leq \eta_R\leq 1$)  
\begin{eqnarray*}
\int_{B_1^c} |x|^{\frac{-\sigma}{p-1}} {F^*}^{\frac{-1}{p-1}}(x)|L_kF^*(x)|^{\frac{p}{p-1}}\,dx&\leq & C (\ln R)^{\frac{-p}{p-1}}\int_{C(\sqrt R,R)}|x|^{\frac{(k-2)p-(\sigma+k)}{p-1}}\eta_R^{\ell-\frac{2p}{p-1}}(x)\,dx\\
&\leq& C (\ln R)^{\frac{-p}{p-1}}\int_{r=\sqrt R}^R r^{\frac{(N+k-2)p-(\sigma+k+N)}{p-1}-1}\,dr.
\end{eqnarray*}
Next, using \eqref{assumptsalpha}, we obtain
$$
\int_{r=\sqrt R}^R r^{\frac{(N+k-2)p-(\sigma+k+N)}{p-1}-1}\,dr=\frac{1}{2}\ln R.
$$
Consequently, we get 
\begin{equation}\label{Ouf}
\int_{B_1^c} |x|^{\frac{-\sigma}{p-1}} {F^*}^{\frac{-1}{p-1}}(x)|L_kF^*(x)|^{\frac{p}{p-1}}\,dx\leq C(\ln R)^{\frac{-1}{p-1}}.	
\end{equation}
Therefore, \eqref{estI2phistar} follows from \eqref{OKvasi} (with $\alpha=0$), \eqref{ayacontinue*} and \eqref{Ouf}.
\end{proof}

\begin{lemma}\label{L2.9}
Let $\mu\in \{\psi,\psi^*\}$. For sufficiently large $T$, $\ell$ and $R$, there holds
\begin{equation}\label{estI1psipsistar}
I_1(\mu)\leq C T^{1-\frac{\alpha+p}{p-1}}\times \left\{\begin{array}{llll}
\ln R+R^{N+k-\frac{\sigma}{p-1}}, &\mbox{if}& 	k>2-N,\\
R^{2-\frac{\sigma}{p-1}}+\ln R, &\mbox{if}& 	k<2-N,\\ 
\ln R\left(R^{2-\frac{\sigma}{p-1}}+\ln R\right), &\mbox{if}& k=2-N.
\end{array}
\right.
\end{equation}
\end{lemma}

\begin{proof}
We shall prove \eqref{estI1psipsistar} for $\mu=\psi$. The case $\mu=\psi^*$ can be treated in a similar way. In view of \eqref{I1phi} and \eqref{testfP12}, we have
\begin{equation}\label{I1psi}
I_1(\psi)=\left(\int_0^\infty t^{\frac{-\alpha}{p-1}} \iota^{\frac{-1}{p-1}}(t) |\iota'(t)|^{\frac{p}{p-1}}\,dt\right)\left(\int_{B_1^c} |x|^{\frac{-\sigma}{p-1}}G(x)\,dx\right).	
\end{equation}
By the definition of the function $G$ and using \eqref{pptsxiR}, we get
\begin{equation}\label{estscdtermI1psi}
\int_{B_1^c} |x|^{\frac{-\sigma}{p-1}}G(x)\,dx	\leq \int_{C(1,2R)}|x|^{\frac{-\sigma}{p-1}}\mathcal{D}(x)\,dx.
\end{equation}
If $k>2-N$, by the defintion of the function $\mathcal{D}$, we obtain
\begin{eqnarray*}
\int_{C(1,2R)}|x|^{\frac{-\sigma}{p-1}}\mathcal{D}(x)\,dx&=& \int_{C(1,2R)}|x|^{k-\frac{\sigma}{p-1}}\left(1-|x|^{2-N-k}\right)\,dx\\
&\leq & C \int_{r=1}^{2R}r^{N-1+k-\frac{\sigma}{p-1}}\,dr\\
&\leq & C \left(\ln R+R^{N+k-\frac{\sigma}{p-1}}\right).
\end{eqnarray*}
If $k<2-N$, we get
\begin{eqnarray*}
\int_{C(1,2R)}|x|^{\frac{-\sigma}{p-1}}\mathcal{D}(x)\,dx&=& \int_{C(1,2R)}|x|^{2-N-\frac{\sigma}{p-1}}\left(1-|x|^{k-2+N}\right)\,dx\\
&\leq & C \int_{r=1}^{2R}r^{1-\frac{\sigma}{p-1}}\,dr\\
&\leq & C \left(R^{2-\frac{\sigma}{p-1}}+\ln R\right).
\end{eqnarray*}
If $k=2-N$, then
\begin{eqnarray*}
\int_{C(1,2R)}|x|^{\frac{-\sigma}{p-1}}\mathcal{D}(x)\,dx&=& \int_{C(1,2R)}|x|^{2-N-\frac{\sigma}{p-1}}\ln |x|\textcolor{red}{dx}\\
&\leq & C \ln R \int_{r=1}^{2R}r^{1-\frac{\sigma}{p-1}}\,dr\\
&\leq & C \ln R\left(R^{2-\frac{\sigma}{p-1}}+\ln R\right).
\end{eqnarray*}
Hence, we deduce that 
\begin{equation}\label{aya3ada}
\int_{C(1,2R)}|x|^{\frac{-\sigma}{p-1}}\mathcal{D}(x)\,dx\leq C 	\left\{\begin{array}{llll}
\ln R+R^{N+k-\frac{\sigma}{p-1}}, &\mbox{if}& 	k>2-N,\\
R^{2-\frac{\sigma}{p-1}}+\ln R, &\mbox{if}& 	k<2-N,\\
\ln R\left(R^{2-\frac{\sigma}{p-1}}+\ln R\right), &\mbox{if}& k=2-N.
\end{array}
\right.
\end{equation}
In view of \eqref{timeest-I1}, \eqref{I1psi} and \eqref{aya3ada}, we obtain \eqref{estI1psipsistar}.
\end{proof}

\begin{lemma}\label{L2.10}
For sufficiently large $T$, $\ell$ and $R$, there holds
\begin{equation}\label{estimation-I2psi}
I_2(\psi)\leq C T^{1-\frac{\alpha}{P-1}}\left\{\begin{array}{llll}
R^{\frac{(k-2)p-(\sigma+k)}{p-1}+N},&\mbox{if}& k>2-N,\\ 
R^{\frac{-(\sigma+2)}{p-1}},&\mbox{if}& k<2-N,\\ 
R^{\frac{-(\sigma+2)}{p-1}}\ln R	,&\mbox{if}& k=2-N.
\end{array}
\right. 
\end{equation}
\end{lemma}

\begin{proof}
In view of \eqref{I2phi} and \eqref{testfP12}, we have
\begin{equation}\label{I1NonI2psi}
I_2(\psi)=\left(\int_0^\infty t^{\frac{-\alpha}{p-1}} \iota(t)\,dt\right)\left(\int_{B_1^c} |x|^{\frac{-\sigma}{p-1}} G^{\frac{-1}{p-1}}(x)|L_kG(x)|^{\frac{p}{p-1}}\,dx\right).
\end{equation}	
By the definition of the function $G$ and using \eqref{pptsxiR},  we obtain 
\begin{equation}\label{integsupportCR2RG}
\int_{B_1^c} |x|^{\frac{-\sigma}{p-1}} G^{\frac{-1}{p-1}}(x)|L_kG(x)|^{\frac{p}{p-1}}\,dx=\int_{C(R,2R)} |x|^{\frac{-\sigma}{p-1}} G^{\frac{-1}{p-1}}(x)|L_kG(x)|^{\frac{p}{p-1}}\,dx.	
\end{equation}
Moreover,  by \eqref{DF-LK} and \eqref{LKD}, elementary calculations show that 
$$
L_kG(x)= \mathcal{D}(x)\Delta(\xi_R^\ell(x))+2\ell \xi_R^{\ell-1}(x)\nabla  \mathcal{D}(x)\cdot \nabla \xi_R(x)-k \ell \xi_R^{\ell-1}(x) \mathcal{D}(x)\nabla \xi_R(x)\cdot \frac{x}{|x|^2},
$$
for every $x\in C(R,2R)$, which yields (in view of \eqref{pptsxiR})
\begin{equation}\label{LKGxest}
|L_kG(x)|\leq  C \xi_R^{\ell-2}(x)\left( R^{-2} \mathcal{D}(x)+R^{-1}|\nabla\mathcal{D}(x)|\right).
\end{equation}
If $k>2-N$, by the definition of the function $\mathcal{D}$ and using \eqref{LKGxest},  we obtain
\begin{eqnarray}\label{samecas1k}
\nonumber  |L_kG(x)|&\leq & C \xi_R^{\ell-2}(x)\bigg(R^{-2}  |x|^k\left(1-|x|^{2-N-k}\right)+R^{-1}\left(|x|^{k-1}+|x|^{1-N}\right) \bigg)\\
 &\leq & C R^{k-2} \xi_R^{\ell-2}(x).
\end{eqnarray}
If $k<2-N$, we get
\begin{eqnarray}\label{samecas2k}
\nonumber  |L_kG(x)|&\leq & C \xi_R^{\ell-2}(x)\bigg(R^{-2}  |x|^{2-N}\left(1-|x|^{k-2+N}\right)+R^{-1}\left(|x|^{k-1}+|x|^{1-N}\right) \bigg)\\
 &\leq & C R^{-N} \xi_R^{\ell-2}(x).
\end{eqnarray}
If $k=2-N$, there holds
\begin{eqnarray}\label{samecas3k}
\nonumber  |L_kG(x)|&\leq & C \xi_R^{\ell-2}(x)\bigg(R^{-2}  |x|^{2-N}\ln |x|+R^{-1}|x|^{1-N}(\ln |x|+1)\bigg)\\
 &\leq & C R^{-N} \ln R\, \xi_R^{\ell-2}(x).
\end{eqnarray}
Hence, by the definition of the function $G$, it follows from  \eqref{samecas1k}, \eqref{samecas2k} and \eqref{samecas3k} that 
$$
|x|^{\frac{-\sigma}{p-1}}G^{\frac{-1}{p-1}}(x)|L_kG(x)|^{\frac{p}{p-1}}\leq C \xi_R^{\ell-\frac{2p}{p-1}}(x)	\left\{\begin{array}{llll}
R^{\frac{(k-2)p}{p-1}}|x|^{\frac{-(\sigma+k)}{p-1}} \left(1-|x|^{2-N-k}\right)^{\frac{-1}{p-1}},&\mbox{if}& k>2-N,\\ 
R^{\frac{-Np}{p-1}} |x|^{\frac{N-2-\sigma}{p-1}} \left(1-|x|^{k-2+N}\right)^{\frac{-1}{p-1}},&\mbox{if}& k<2-N,\\ 
R^{\frac{-Np}{p-1}}(\ln R)^{\frac{p}{p-1}} |x|^{\frac{N-2-\sigma}{p-1}} (\ln |x|)^{\frac{-1}{p-1}},&\mbox{if}& k=2-N,	
\end{array}
\right.
$$
which yields (since $0\leq \xi_R\leq 1$)
\begin{equation}\label{JLpense}
\int_{C(R,2R)} |x|^{\frac{-\sigma}{p-1}} G^{\frac{-1}{p-1}}(x)|L_kG(x)|^{\frac{p}{p-1}}\,dx	\leq C  \left\{\begin{array}{llll}
R^{\frac{(k-2)p-(\sigma+k)}{p-1}+N},&\mbox{if}& k>2-N,\\ 
R^{\frac{-(\sigma+2)}{p-1}},&\mbox{if}& k<2-N,\\ 
R^{\frac{-(\sigma+2)}{p-1}}\ln R	,&\mbox{if}& k=2-N.
\end{array}
\right.
\end{equation}
Thus, \eqref{estimation-I2psi} follows from \eqref{OKvasi}, \eqref{I1NonI2psi} and \eqref{JLpense}.
 \end{proof}

\begin{lemma}\label{L2.11}
Let  $\alpha=0$, $k>2-N$ and $p(N+k-2)=N+\sigma+k$.  For sufficiently large $T$, $\ell$ and $R$, there holds
\begin{equation}\label{estI2psi*}
I_2(\psi^*)\leq CT(\ln R)^{\frac{-1}{p-1}}. 
\end{equation} 	
\end{lemma}

\begin{proof}
In view of \eqref{DF-LK}, \eqref{I2phi}  and \eqref{testfP12-cr}, we get 
\begin{equation}\label{ayacontinue}
I_2(\varphi)=\left(\int_0^\infty \iota(t)\,dt\right)\left(\int_{B_1^c} |x|^{\frac{-\sigma}{p-1}} {G^*}^{\frac{-1}{p-1}}(x)|L_kG^*(x)|^{\frac{p}{p-1}}\,dx\right).
\end{equation}	
By the definition of the function $G^*$ and using \eqref{pptsetaRII},   elementary calculations show that 
$$
|L_kG^*(x)| \leq C \eta_R^{\ell-2}(x)(\ln R)^{-1}\left(|x|^{-2}\mathcal{D}(x)+|x|^{-1}|\nabla \mathcal{D}(x)|\right),\quad x\in C(\sqrt{R},R).
$$
Since $k>2-N$, by the definition of the function $\mathcal{D}$, we deduce that for all $x\in C(\sqrt{R},R)$, 
$$
|L_kG^*(x)| \leq C \frac{|x|^{k-2}}{\ln R}\eta_R^{\ell-2}(x),
$$
which yields (in view of \eqref{pptsetaRI}) 
\begin{eqnarray*}
\int_{B_1^c} |x|^{\frac{-\sigma}{p-1}} {G^*}^{\frac{-1}{p-1}}(x)|L_kG^*(x)|^{\frac{p}{p-1}}\,dx&=& \int_{C(\sqrt R,R)} |x|^{\frac{-\sigma}{p-1}} {G^*}^{\frac{-1}{p-1}}(x)|L_kG^*(x)|^{\frac{p}{p-1}}\,dx	\\
&\leq & C (\ln R)^{\frac{-p}{p-1}}\int_{C(\sqrt R,R)} |x|^{\frac{(k-2)p-(\sigma+k)}{p-1}}\left(1-|x|^{2-N-k}\right)^{\frac{-1}{p-1}}\\
&\leq &  C(\ln R)^{\frac{-p}{p-1}} \int_{r=\sqrt R}^Rr^{\frac{(N+k-2)p-(N+\sigma+k)}{p-1}-1}\,dr.
\end{eqnarray*}
Since $p(N+k-2)=N+\sigma+k$, we deduce that 
\begin{equation}\label{IYAH}
\int_{B_1^c} |x|^{\frac{-\sigma}{p-1}} {G^*}^{\frac{-1}{p-1}}(x)|L_kG^*(x)|^{\frac{p}{p-1}}\,dx\leq C 	(\ln R)^{\frac{-1}{p-1}}. 
\end{equation}
Thus, \eqref{OKvasi} (with $\alpha=0$), \eqref{ayacontinue} and \eqref{IYAH} give us \eqref{estI2psi*}.
\end{proof}

\section{Proof of the main results}\label{sec3}

This section is devoted to the proofs of our main results. 

\begin{proof}[Proof of Theorem \ref{T1}] (i) 
We argue by contradiction by assuming that 	$u\in L^p_{loc}(\Omega)$ is a global weak solution to problem \eqref{P} under the boundary condition \eqref{BC1}. Then, by Proposition \ref{PR1} and Lemma \ref{L1}, there holds
\begin{equation}\label{gestPrT1}
\int_\Gamma  f\varphi\, dS_x\,dt\leq C \sum_{i=1}^2 I_i(\varphi),	
\end{equation}
where the function $\varphi$ is defined by \eqref{testfP11}.  On the other hand, one has
\begin{eqnarray}\label{integrbdT1}
\nonumber \int_\Gamma  f\varphi\, dS_x\,dt &=& \left(\int_0^\infty \iota(t)\,dt\right)	\left(\int_{\partial B_1} f(x) F(x)\,dS_x\right)\\
\nonumber  &=& \left(\int_0^\infty \iota(t)\,dt\right)	\left(\int_{\partial B_1}f(x) \xi_R^\ell(x)\,dS_x\right)\\
\mbox{(by \eqref{OKvasi} and \eqref{pptsxiR})}&=& CT I_f.
\end{eqnarray}
Hence, by \eqref{gestPrT1}, \eqref{integrbdT1}, Lemma \ref{L2.7} and \eqref{estI2phi}, we obtain
\begin{equation}\label{wrkprT1}
I_f\leq C 	\left(T^{-\frac{\alpha+p}{p-1}}\left(\ln R+ R^{N+k-\frac{\sigma}{p-1}}\right)+T^{-\frac{\alpha }{p-1}}R^{\frac{(k-2)p-(\sigma+k)}{p-1}+N}\right). 
\end{equation}
Thus, fixing $R$ and passing to the limit as $T\to \infty$ in \eqref{wrkprT1}, we get (since $\alpha>0$) $I_f\leq 0$, which contradicts the condition $I_f>0$. This proves part (i) of Theorem \ref{T1}.

\medskip
\noindent (ii) As previously, assume that   $u\in L^p_{loc}(\Omega)$ is a global weak solution to problem \eqref{P} under the boundary condition \eqref{BC3}. Then, by Proposition \ref{PR2} and Lemma \ref{L3}, there holds
\begin{equation}\label{gestPrT1ii}
-\int_\Gamma  f \frac{\partial \psi}{\partial \nu}\, dS_x\,dt\leq C \sum_{i=1}^2 I_i(\psi),	
\end{equation}
where the function $\psi$ is defined by \eqref{testfP12}. On the other hand, by \eqref{useful-Dirichlet-ND}, one has
\begin{eqnarray}\label{integrbdT1ii}
\nonumber -\int_\Gamma  f \frac{\partial \psi}{\partial \nu}\, dS_x\,dt &=& C \int_\Gamma   \iota(t) f(x)\,dS_x\,dt\\
&=& CT I_f. 
\end{eqnarray}
Hence, by \eqref{gestPrT1ii}, \eqref{integrbdT1ii}, Lemma \ref{L2.9} and Lemma \ref{L2.10}, we obtain
\begin{equation}\label{Iktibha}
I_f\leq C \left(T^{-\frac{\alpha+p}{p-1}}A(R)+	T^{-\frac{\alpha}{p-1}}B(R)\right),
\end{equation}
where 
$$
A(R)=\left\{\begin{array}{llll}
\ln R+R^{N+k-\frac{\sigma}{p-1}}, &\mbox{if}& 	k>2-N,\\
R^{2-\frac{\sigma}{p-1}}+\ln R, &\mbox{if}& 	k<2-N,\\ 
\ln R\left(R^{2-\frac{\sigma}{p-1}}+\ln R\right), &\mbox{if}& k=2-N
\end{array}
\right.
$$
and
$$
B(R)=\left\{\begin{array}{llll}
R^{\frac{(k-2)p-(\sigma+k)}{p-1}+N},&\mbox{if}& k>2-N,\\ 
R^{\frac{-(\sigma+2)}{p-1}},&\mbox{if}& k<2-N,\\ 
R^{\frac{-(\sigma+2)}{p-1}}\ln R	,&\mbox{if}& k=2-N.
\end{array}
\right. 
$$
Thus, fixing $R$ and passing to the limit as $T\to \infty$ in \eqref{Iktibha}, we get a contradiction with  $I_f>0$. This proves part (ii) of Theorem \ref{T1}. 
\end{proof}

\begin{proof}[Proof of Theorem \ref{T2}](i) Suppose that $u\in L^p_{loc}(\Omega)$ is a global weak solution to problem \eqref{P} under the boundary condition \eqref{BC1}. Then, by the proof of part (i) of Theorem \ref{T1}, we get \eqref{wrkprT1}. Taking $T=R^\theta$, $\theta>0$, \eqref{wrkprT1} reduces to 
\begin{equation}\label{bahiyeser}
I_f\leq C 	\left(R^{-\frac{\theta(\alpha+p)}{p-1}}\ln R+ R^{N+k-\frac{\sigma}{p-1}-\frac{\theta(\alpha+p)}{p-1}}+R^{\frac{(k-2+N)p-(\sigma+k+N)-\theta \alpha}{p-1}}\right).
\end{equation}	
On the other hand, observe that for $\theta=2$, one has
$$
N+k-\frac{\sigma}{p-1}-\frac{\theta(\alpha+p)}{p-1}=\frac{(k-2+N)p-(\sigma+k+N)-\theta \alpha}{p-1}=\frac{p(k-2+N)-(\sigma+k+N+2\alpha)}{p-1}.
$$
Hence, for $\theta=2$, \eqref{bahiyeser} reduces to  
\begin{equation}\label{hathikahiya}
I_f\leq C 	 \left(	R^{-\frac{2(\alpha+p)}{p-1}}\ln R+R^{\frac{p(k-2+N)-(\sigma+k+N+2\alpha)}{p-1}}\right).
\end{equation}
Thus, using that $\alpha+p>0$ and \eqref{blowup-cd}, by passing to the limit as $R\to \infty$ in \eqref{hathikahiya}, we obtain a contradiction with $I_f>0$. This proves part (i) of Theorem \ref{T2}. 

\medskip  \noindent(ii) Always we use the contradiction argument by supposing that $u\in L^p_{loc}(\Omega)$ is a global weak solution to problem \eqref{P} under the boundary condition \eqref{BC3}. Then, by the proof of part (ii) of Theorem \ref{T1}, we get \eqref{Iktibha}, which reduces to (since $k>2-N$)
$$
I_f\leq C \left(T^{-\frac{\alpha+p}{p-1}}\left(\ln R+R^{N+k-\frac{\sigma}{p-1}}\right)+	T^{-\frac{\alpha}{P-1}}R^{\frac{(k-2)p-(\sigma+k)}{p-1}+N}\right).
$$
As in the proof of part (i), taking $T=R^2$, the above estimate reduces to \eqref{hathikahiya}. Hence, we reach a contradiction with $I_f>0$. This proves part (ii) of Theorem \ref{T2}. 
\end{proof}

\begin{proof}[Proof of Theorem \ref{T3}](i) Let $u\in L^p_{loc}(\Omega)$ be a global weak solution to problem \eqref{P} under the boundary condition \eqref{BC1}. Then \eqref{hathikahiya} holds. Since $k\leq 2-N$ and $\sigma+2(\alpha+1)>2-N-k$, then 
$$
p(k-2+N)\leq 0<\sigma+k+N+2\alpha. 
$$
Hence, passing to the limit as $R\to \infty$ in \eqref{hathikahiya}, we obtain a contradiction with $I_f>0$. This proves part (i) of Theorem \ref{T3}. 

\medskip\noindent(ii) Assume that  $u\in L^p_{loc}(\Omega)$ be a global weak solution to problem \eqref{P} under the boundary condition \eqref{BC3}.  Then, by the proof of part (ii) of Theorem \ref{T1}, we get \eqref{Iktibha}, which reduces to (since $k\leq 2-N$)
$$
I_f\leq C \left(T^{-\frac{\alpha+p}{p-1}}\ln R\left(R^{2-\frac{\sigma}{p-1}}+\ln R\right)+	T^{-\frac{\alpha}{p-1}}R^{\frac{-(\sigma+2)}{p-1}}\ln R	\right).
$$ 
Taking $T=R^2$, the above inequality reduces to 
\begin{equation}\label{9rib}
I_f\leq C \left(R^{\frac{-\sigma-2(\alpha+1)}{p-1}}\left(\ln R+1\right)+(\ln R)^2 R^{\frac{-2(\alpha+p)}{p-1}}\right).	
\end{equation}
Since $\alpha+p>0$ and $\sigma+2(\alpha+1)>0$, passing to the limit as $R\to \infty$ in \eqref{9rib}, we reach a contradiction with $I_f>0$. This proves part (ii) of Theorem \ref{T3}.
\end{proof}

\begin{proof}[Proof of Theorem \ref{T4}](i) Let 	$u\in L^p_{loc}(\Omega)$ be a global weak solution to problem \eqref{P} under the boundary condition \eqref{BC1}. Then, by Proposition \ref{PR1} and Lemma \ref{L2}, there holds
\begin{equation}\label{gestPrT4}
\int_\Gamma  f\varphi^*\, dS_x\,dt\leq C \sum_{i=1}^2 I_i(\varphi^*),	
\end{equation}
where the function $\varphi^*$ is defined by \eqref{testfP11-cr}. 	On the other hand, one has
\begin{eqnarray}\label{integrbdT4}
\nonumber \int_\Gamma  f\varphi^*\, dS_x\,dt &=& \left(\int_0^\infty \iota(t)\,dt\right)	\left(\int_{\partial B_1} f(x) F^*(x)\,dS_x\right)\\
\nonumber  &=& \left(\int_0^\infty \iota(t)\,dt\right)	\left(\int_{\partial B_1}f(x) \eta_R^\ell(x)\,dS_x\right)\\
\mbox{(by \eqref{OKvasi} and \eqref{pptsetaRI})}&=& CT I_f.
\end{eqnarray}
Hence, by \eqref{gestPrT4}, \eqref{integrbdT4}, Lemma \ref{L2.7} (with $\alpha=0$) and \eqref{estI2phistar}, we obtain
\begin{equation}\label{wrkprT4}
I_f\leq C 	\left(T^{-\frac{p}{p-1}}\left(\ln R+ R^{N+k-\frac{\sigma}{p-1}}\right)+(\ln R)^{\frac{-1}{p-1}}\right). 
\end{equation}
Hence, taking $T=R^\theta$ with  $\theta>\max\left\{0,\frac{(p-1)(N+k)-\sigma}{p}\right\}$ and passing to the limit as $R\to \infty$ in \eqref{wrkprT4}, we obtain a contradiction with $I_f>0$. This proves part (i) of Theorem \ref{T4}. 

\medskip \noindent(ii) Now, suppose that $u\in L^p_{loc}(\Omega)$ be a global weak solution to problem \eqref{P} under the boundary condition \eqref{BC3}. Then, by Proposition \ref{PR2} and Lemma \ref{L4}, there holds
\begin{equation}\label{gestPrT4ii}
-\int_\Gamma  f \frac{\partial \psi^*}{\partial \nu}\, dS_x\,dt\leq C \sum_{i=1}^2 I_i(\psi^*)
\end{equation}
where the function $\psi^*$ is defined by \eqref{testfP12-cr}. On the other hand, by \eqref{useful-Dirichlet-ND-cr} and \eqref{OKvasi} (with $\alpha=0$), we have 
\begin{eqnarray}\label{9ribbarcha}
\nonumber -\int_\Gamma  f \frac{\partial \psi^*}{\partial \nu}\, dS_x\,dt&=&C \int_\Gamma  \iota(t)f(x)\,dS_x\,dt\\
&=& CTI_f.
\end{eqnarray}
Hence, by \eqref{blowup-cd-cr}, \eqref{gestPrT4ii}, \eqref{9ribbarcha},  Lemma \ref{L2.9} and Lemma \ref{L2.11}, we obtain \eqref{wrkprT4}. The rest of the proof is the same as that of part (i). Then, the proof of part (ii) of Theorem \ref{T4} is completed.
\end{proof}

\section*{Acknowledgements}

The second author is supported by Researchers Supporting Project number (RSP--2021/4), King Saud University, Riyadh, Saudi Arabia.
The third author is supported by the National Natural Science Foundation of China (No.11501303), and Tianjin Natural Science Foundation (No.19JCQNJC14600).

\end{document}